\documentclass[11pt,reqno]{amsart}
\usepackage{fixltx2e}
\usepackage[utf8x]{inputenc}
\usepackage[T1]{fontenc}
\usepackage[paper=a4paper]{geometry} 
\usepackage{mathtools}
\usepackage[frenchstyle,narrowiints,partialup]{kpfonts}
\usepackage{bm}
\usepackage{amsthm}
\usepackage{hyperref}
\usepackage{enumerate}
\usepackage[UKenglish]{babel}
\usepackage{microtype}
\hypersetup{%
   unicode=true,
   hidelinks,%
   pdftitle={Poisson structures and star products on quasimodular forms},%
   pdfauthor={\textcopyright\ Fran{\c c}ois Dumas et Emmanuel Royer (for Universit\'e Blaise Pascal, laboratoire de math\'ematiques)},%
   pdfsubject={Quasimodular forms, Poisson brackets, Rankin-Cohen brackets, formal deformation, Eholzer product, star product},%
   pdfkeywords={Quasimodular, forms, Poisson, brackets, Rankin-Cohen, deformation, star, product, derivation, Eholzer, star, Serre's derivative, kernel, Jacobian Poisson structure},%
   baseurl={http://math.univ-bpclermont.fr/~royer/},%
   pdflang=en-GB,
  }
\newtheoremstyle{erthm}% name
  {}%      Space above, empty = `usual value'
  {}%      Space below
  {\itshape}% Body font
  {}%         Indent amount (empty = no indent, \parindent = para indent)
  {\fontseries{bx}\selectfont\itshape}% Thm head font
  {--}%        Punctuation after thm head
  { }%     Space after thm head: " " = normal interword space;
  {}% Thm head spec
  \newtheoremstyle{errem}% name
    {}%      Space above, empty = `usual value'
    {}%      Space below
    {}% Body font
    {}%         Indent amount (empty = no indent, \parindent = para indent)
    {\ttfamily\itshape}% Thm head font
    {--}%        Punctuation after thm head
    { }%     Space after thm head: " " = normal interword space;
    {}% Thm head spec
\newtheoremstyle{erdfn}% name
  {}%      Space above, empty = `usual value'
  {}%      Space below
  {\itshape}% Body font
  {}%         Indent amount (empty = no indent, \parindent = para indent)
  {\fontseries{bx}\selectfont\itshape}% Thm head font
  {--}%        Punctuation after thm head
  { }%     Space after thm head: " " = normal interword space;
  {}% Thm head spec
\theoremstyle{erthm}
\newtheorem{theoremintro}{Theorem}

\newtheorem{propintro}{Proposition}

\newtheorem{theorem}{Theorem}
\newtheorem{proposition}[theorem]{Proposition}
\newtheorem{lemma}[theorem]{Lemma}
\newtheorem{corollary}[theorem]{Corollary}
\theoremstyle{errem}
\newtheorem*{remarkintro}{Remark}
\newtheorem*{remark}{Remark}
\theoremstyle{erdfn}
\newtheorem{definition}{Definition}
\DeclarePairedDelimiter{\accol}{\{}{\}}%..................................................{.,.}
\DeclareMathOperator{\bi}{b}%..................................................................application.b
\DeclareMathOperator{\Bi}{B}%..................................................................application.B
\newcommand*{\boldmu}{\bm{\mu}}
\newcommand*{\CC}{\mathbb{C}}%.........................................................corps.complexe
\DeclarePairedDelimiter{\crochet}{\lbrack}{\rbrack}%...............................[.,.]
\DeclareMathOperator{\curl}{curl}%............................................................curl
\newcommand*{\dd}{%
  \mathop{\mathrm{d}\null}\mskip-\thinmuskip\mathord{\null}}%.............differential.d
\DeclareMathOperator{\der}{d}%................................................................derivation.d
\DeclareMathOperator{\Diff}{D}%................................................................derivation.D
\DeclareMathOperator{\dV}{v}%.................................................................derivation.v
\DeclareMathOperator{\dZ}{w}%.................................................................derivation.Z
\DeclareMathOperator{\E}{E}%....................................................................serie.Eisenstein
\newcommand*{\eexp}{\mathrm{e}}%.......................................................e.exponentiel
\renewcommand*{\epsilon}{\varepsilon}%..................................................epsilon
\newcommand*{\fQM}{\Mod_{*}^{\leq\infty}}%........................................algebre.des.quasimodulaires
\newcommand*{\fMod}{\Mod_{*}}%...........................................................algebre.des.modulaires
\DeclareMathOperator{\grad}{grad}%........................................................grad
\DeclareMathOperator{\HP}{HP}%........................................................Poisson.center
\newcommand*{\ic}{\mathrm{i}}%............................................................i.complexe
\DeclareMathOperator{\jac}{jac}%........................................................jac
\newcommand*{\Mod}{\mathcal{M}}%......................................................modulaires
\newcommand*{\N}{\mathbb{N}}%...........................................................ensemble.entiers.naturels
\DeclarePairedDelimiter{\parenth}{\lparen}{\rparen}%.............................(.,.)
\newcommand*{\pk}{\mathcal{H}}%.........................................................poincare
\newcommand*{\Q}{\mathbb{Q}}%...........................................................corps.rationnel
\DeclareMathOperator{\RC}{RC}%...............................................................Rankin.Cohen.bracket
\DeclarePairedDelimiter{\scal}{\langle}{\rangle}%.....................................<.,.>
\newcommand*{\sldz}{\mathrm{SL}(2,\mathbb{Z})}%............................SL(2,Z)
\DeclareMathOperator{\X}{X}%..................................................................X
\newcommand*{\Z}{\mathbb{Z}}%...........................................................anneau.entiers
\title[Star products on quasimodular forms]{Poisson structures and star products on quasimodular forms}
\thanks{We thank François Martin and Anne Pichereau for many fruitful discussions. }
\author[F. Dumas]{Fran{\c c}ois Dumas}
\address{%
Fran{\c c}ois Dumas\\
Clermont Universit\'e\\
Universit\'e Blaise Pascal\\
Laboratoire de math\'ematiques\\
BP 10448\\
F-63000 Clermont-Ferrand\\
France %
}
\curraddr{%
Universit\'e Blaise Pascal\\
Laboratoire de math\'ematiques\\
Les C\'ezeaux\\
BP 80026\\
F-63171 Aubi\`ere Cedex\\
France %
}
\email{{francois.dumas@math.univ-bpclermont.fr}}
\author[E. Royer]{Emmanuel Royer}
\address{%
Emmanuel Royer\\
Clermont Universit\'e\\
Universit\'e Blaise Pascal\\
Laboratoire de math\'ematiques\\
BP 10448\\
F-63000 Clermont-Ferrand\\
France %
}
\curraddr{%
Universit\'e Blaise Pascal\\
Laboratoire de math\'ematiques\\
Les C\'ezeaux\\
BP 80026\\
F-63171 Aubi\`ere Cedex\\
France %
}
\email{{emmanuel.royer@math.univ-bpclermont.fr}}
\begin{document}
\mathtoolsset{showonlyrefs,mathic,centercolon}
%..............................
\begin{abstract}
We construct and classify all Poisson structures on quasimodular forms that extend the one coming from the first Rankin-Cohen bracket on the modular forms. We use them to build formal deformations on the algebra of quasimodular forms. %
\end{abstract}
\maketitle
\tableofcontents%
%.............................................
\section{Introduction}
In~\cite[Theorem 7.1]{MR0382192}, Henri Cohen defined a collection of bi-differential operators on modular forms. Let \(n\) be a positive integer, \(f\) a modular form of weight \(k\) and \(g\) a modular form of weight \(\ell\). The \(n\)-th Rankin-Cohen bracket of \(f\) and \(g\) is the modular form of weight \(k+\ell+2n\) defined by %
\[%
\RC_n(f,g)=\sum_{r=0}^n(-1)^r\binom{k+n-1}{n-r}\binom{\ell+n-1}{r}\Diff^r f\Diff^{n-r}g, \qquad \left(\Diff=\frac{1}{2\pi\ic}\frac{\dd}{\dd z}\right).
\]
The algebraic structure of these brackets has been studied in Zagier's seminal work \cite{MR1280058}. That Rankin-Cohen brackets define a formal deformation of the algebra of modular forms has been widely studied. Important contributions are~\cite{MR1403254,MR1418868,Yao_PHD,MR2319770,MR3025141,2013arXiv1301.2111K}.

In this paper, we construct formal deformations of the algebra \(\fQM\) of quasimodular forms. This algebra is generated over \(\CC\) by the three Eisenstein series \(\E_2\), \(\E_4\) and \(\E_6\). The algebra \(\Mod_*\) of modular forms is the subalgebra generated by \(\E_4\) and \(\E_6\). As a first step, we classify the admissible Poisson structures of \(\fQM\).  A Poisson bracket \(\accol{\phantom{f},\phantom{g}}\) on \(\fQM\) is admissible if %
\begin{enumerate}[i)]
\item the restriction of \(\accol{\phantom{f},\phantom{g}}\) to the algebra \(\fMod\) of modular forms is the first Rankin-Cohen bracket \(\RC_1\)
\item it satisfies \(\accol{\Mod_{k}^{\leq s},\Mod_{\ell}^{\leq t}}\subset\Mod_{k+\ell+2}^{\leq s+t}\) for any even integers \(k,\ell\) and any integers \(s\) and \(t\) 
\end{enumerate}
where \(\Mod_{k}^{\leq s}\) is the vector space of quasimodular forms of weight \(k\) and depth less than \(s\). The vector space of parabolic modular forms of weight \(12\) is one dimensional. We choose \(\Delta=\E_4^3-\E_6^2\) a generator. %

\begin{propintro}[First family of Poisson brackets]\label{prop_A}
For any \(\lambda\in\CC^*\), there exists an admissible Poisson bracket \(\accol{\phantom{f},\phantom{g}}_\lambda\) on the algebra of quasimodular forms defined by the following values on the generators: %
\[%
\accol{\E_4,\E_6}_\lambda=-2\Delta,\quad \accol{\E_2,\E_4}_\lambda=-\frac{1}{3}\left(2\E_6\E_2-\lambda\E_4^2\right),\quad \accol{\E_2,\E_6}_\lambda=-\frac{1}{2}\left(2\E_4^2\E_2-\lambda\E_4\E_6\right). %
\]
Moreover, %
\begin{enumerate}[i)]
\item for any \(\lambda\in\CC^*\), the Poisson bracket \(\accol{\phantom{f},\phantom{g}}_\lambda\) is not unimodular. 
\item\label{prop_Apoinii} The Poisson algebras \(\left(\fQM,\accol{\phantom{f},\phantom{g}}_\lambda\right)\) and \(\left(\fQM,\accol{\phantom{f},\phantom{g}}_{\lambda'}\right)\) are Poisson modular iso\-mor\-phic for all \(\lambda\) and \(\lambda'\) in \(\CC^*\).
\item For any \(\lambda\in\CC^*\), the Poisson center of \(\left(\fQM,\accol{\phantom{f},\phantom{g}}_\lambda\right)\) is \(\CC\). %
\end{enumerate}
\end{propintro}
\begin{remarkintro}
A Poisson isomorphism \(\varphi\) on \(\fQM\) is modular if \(\varphi(\fMod)\subset\fMod\). %
\end{remarkintro}
Thanks to~\ref{prop_Apoinii}) in Proposition~\ref{prop_A}, we restrict to the bracket \(\accol{\phantom{f},\phantom{g}}_1\). Following~\cite[Eq. (38)]{MR1280058}, we consider the derivation \(\dZ\) on \(\fQM\) defined by %
\[%
\dZ(f)=\frac{\accol{\Delta,f}_1}{12\Delta}. %
\]
A derivation \(\delta\) on \(\fQM\) is complex-like if \(\delta\left(\Mod_k^{\leq s}\right)\subset\Mod_{k+2}^{\leq s+1}\) for any \(k\) and \(s\). 
The set of complex-like derivations \(\delta\) such that \(kf\delta(g)-\ell g\delta(f)=0\) for any \(f\in\Mod_k^{\leq s}\) and \(g\in\Mod_\ell^{\leq t}\), for any \(k,\ell,s,t\), is a one dimensional vector space over \(\CC\). Let \(\pi\) be a generator.  The following Theorem provides a first family of formal deformations of the algebra \(\fQM\). %
\begin{theoremintro}\label{thm_A}
For any \(a\in\CC\), let \(\der_a\) be the derivation on \(\fQM\) defined by \(\der_a=a\pi+\dZ\). %
\begin{enumerate}[i)]
\item\label{item_thmAi} For all quasimodular forms \(f\) and \(g\) of respective weights \(k\) and \(\ell\), we have %
\[%
\accol{f,g}_1=kf\der_a(f)-\ell g\der_a(f). %
\]
\item\label{item_thmAii} More generally, for any \(a\in\CC\), the brackets defined for any integer \(n\geq 0\) by %
\[%
\crochet{f,g}_{\der_a,n}=\sum_{r=0}^n(-1)^r\binom{k+n-1}{n-r}\binom{\ell+n-1}{r}\der_a^r(f)\der_a^{n-r}(g),\quad (f\in\Mod_k^{\leq\infty}, g\in\Mod_\ell^{\leq\infty})
\]
satisfy %
\[%
\crochet*{\Mod_k^{\leq\infty},\Mod_\ell^{\leq\infty}}_{\der_a,n}\subset\Mod_{k+\ell+2n}^{\leq\infty}
\]
and define a formal deformation of \(\fQM\). %
\item\label{item_thmAiii} Moreover, \(\crochet*{\Mod_k^{\leq s},\Mod_\ell^{\leq t}}_{\der_a,n}\subset\Mod_{k+\ell+2n}^{\leq s+t}\) for all \(n,s,t,k,\ell\) if and only if \(a=0\). % 
\end{enumerate}
\end{theoremintro}
\begin{remarkintro}
A generator is \(\pi\) defined by linear extension of %
\(%
\pi(f)=kf\E_2 %
\)
for \(f\) any quasimodular form of weight \(k\). For this choice, the derivation \(\der_a\) is defined on the generators by %
\[
\der_a\E_2=2a\E_2^2-\frac{1}{12}\E_4,\qquad \der_a\E_4=4a\E_4\E_2-\frac{1}{3}\E_6,\qquad \der_a\E_6=6a\E_6\E_2-\frac{1}{2}\E_4^2. %
\]
\end{remarkintro}

We construct a second family of Poisson brackets.
\begin{propintro}[Second family of Poisson brackets]\label{prop_B}
For any \(\alpha\in\CC\), there exists an admissible Poisson bracket \(\parenth{\phantom{f},\phantom{g}}_\alpha\) on the algebra of quasimodular forms defined by the following values on the generators: %
\[%
\parenth{\E_4,\E_6}_\alpha=-2\Delta,\qquad \parenth{\E_2,\E_4}_\alpha=\alpha\E_6\E_2,\qquad \parenth{\E_2,\E_6}_\alpha=\frac{3}{2}\alpha\E_4^2\E_2. %
\]
Moreover, %
\begin{enumerate}[i)]
\item for any \(\alpha\in\CC\setminus\{4\}\), the Poisson bracket \(\parenth{\phantom{f},\phantom{g}}_\alpha\) is not unimodular. For \(\alpha=4\), the Poisson bracket  \(\parenth{\phantom{f},\phantom{g}}_4\) is Jacobian (hence unimodular) of potential \(k_0=-2\Delta\E_2\). %
\item The Poisson algebras \(\left(\fQM,\parenth{\phantom{f},\phantom{g}}_\alpha\right)\) and \(\left(\fQM,\parenth{\phantom{f},\phantom{g}}_{\alpha'}\right)\) are Poisson modular iso\-mor\-phic if and only if \(\alpha=\alpha'\). %
\item For any \(\alpha\in\CC\), %
\begin{enumerate}[\indent a)]
\item if \(\alpha\notin\Q\), the Poisson center of  \(\left(\fQM,\parenth{\phantom{f},\phantom{g}}_\alpha\right)\) is \(\CC\) %
\item if \(\alpha=0\), the Poisson center of  \(\left(\fQM,\parenth{\phantom{f},\phantom{g}}_\alpha\right)\) is \(\CC[\E_2]\) %
\item if \(\alpha=p/q\) with \(p\in\Z^*\) and \(q\in\N^*\), \(p\) and \(q\) coprimes, the Poisson center of  \(\left(\fQM,\parenth{\phantom{f},\phantom{g}}_\alpha\right)\) is%
\[%
\begin{cases}
\CC & \text{if \(p<0\)}\\
\raisebox{3ex}{\makebox[0cm]{}}
\CC\left[\Delta^p\E_2^{4q}\right] & \text{if \(p\geq 1\) is odd}\\
\raisebox{3ex}{\makebox[0cm]{}}
\CC\left[\Delta^u\E_2^{2q}\right] & \text{if \(p=2u\) for odd \(u\geq 1\)}\\
\raisebox{3ex}{\makebox[0cm]{}}
\CC\left[\Delta^v\E_2^{q}\right] & \text{if \(p=4v\) with \(v\geq 1\).}
\end{cases}
\]
\end{enumerate}
\end{enumerate}
\end{propintro}
\begin{remarkintro}
The bracket  \(\parenth{\phantom{f},\phantom{g}}_0\) is the trivial bracket. %
\end{remarkintro}
This second family provides a new set of formal deformations of the algebra of quasimodular forms. Following~\cite[Eq. (38)]{MR1280058}, we consider the derivation \(\dV\) defined on \(\fQM\) by %
\[%
\dV(f)=\frac{\parenth{\Delta,f}_\alpha}{12\Delta}. %
\]
Let us define \(\mathcal{K}_\alpha\colon\fQM\to\CC\) by setting \(\mathcal{K}_\alpha(f)=k-(3\alpha+2)s\) if \(f\) has weight \(k\) and depth \(s\).  The set of complex-like derivations \(\delta\) such that \(\mathcal{K}_\alpha(f)f\delta(g)-\mathcal{K}_\alpha(g)g\delta(f)=0\) for any \(f\in\Mod_k^{\leq s}\) and \(g\in\Mod_\ell^{\leq t}\), for any \(k,\ell,s,t\), is a one dimensional vector space over \(\CC\). Let \(\pi_\alpha\) be a generator. 
We note %
\[%
\Mod_k^s=\Mod_k\E_2^s. %
\]
\begin{theoremintro}\label{thm_B}
Let \(\alpha\in\CC\). For any \(b\in\CC\), let \(\delta_{\alpha,b}\) be the derivation on \(\fQM\) defined by \(\delta_{\alpha,b}=b\pi_{\alpha}+\dV\). %
\begin{enumerate}[i)]
\item\label{item_thmBi} For all \(f\in\Mod_k^{s}\) and \(g\in\Mod_\ell^{t}\), we have %
\[%
\parenth{f,g}_\alpha=\left(k-(3\alpha+2)s\right)f\delta_{\alpha,b}(g)-\left(\ell-(3\alpha+2)t\right)g\delta_{\alpha,b}(f) %
\]
for any \(f\in\Mod_k^s\) and \(g\in\Mod_\ell^t\). %
\item\label{item_thmBii} Moreover, the brackets defined for any integer \(n\geq 0\) by %
\[%
\crochet{f,g}_{\delta_{\alpha,b},n}^{\mathcal{K}_\alpha}=\sum_{r=0}^n(-1)^r\binom{k-(3\alpha+2)s+n-1}{n-r}\binom{\ell-(3\alpha+2)t+n-1}{r}\delta_{\alpha,b}^r(f)\delta_{\alpha,b}^{n-r}(g),
\]
for any \(f\in\Mod_k^{s}\) and \(g\in\Mod_\ell^{t}\) define a formal deformation of \(\fQM\) satisfying %
\[%
\crochet*{\Mod_k^{\leq s},\Mod_\ell^{\leq t}}_{\delta_{\alpha,b},n}^{\mathcal{K}_\alpha}\subset\Mod_{k+\ell+2n}^{\leq s+t}
\]
for all \(k,\ell\) in \(2\N\) and \(s,t\) in \(\N\) if and only if \(b=0\). 
\end{enumerate}
\end{theoremintro}
\begin{remarkintro}
A generator \(\pi_\alpha\) is defined by linear extension of: %
\[%
\pi_\alpha(f)=[k-(3\alpha+2)s]f\E_2\qquad (f\in\Mod_k^s). %
\]
For this choice, the derivation \(\delta_{\alpha,b}\) is defined on the generators by %
\[
\delta_{\alpha,b}\E_2=-3b\alpha\E_2^2,\qquad \delta_{\alpha,b}\E_4=4b\E_4\E_2-\frac{1}{3}\E_6,\qquad \delta_{\alpha,b}\E_6=6b\E_6\E_2-\frac{1}{2}\E_4^2. %
\]
\end{remarkintro}
To complete the classification of Poisson structures, we introduce a third family of Poisson brackets. We note however that, when \(\mu\neq 0\),  this third family does not lead to a formal deformation of \(\fQM\) with the shape of Rankin-Cohen brackets (see~\S\ref{sec_extfamtrois}).%
\begin{propintro}[Third family of Poisson brackets]\label{prop_C}%
For any \(\mu\in\CC\), there exists an admissible Poisson brackets \(\scal{\phantom{f},\phantom{f}}_\mu\) on the algebra of quasimodular forms defined by the following values on the generators: %
\[%
\scal{\E_4,\E_6}_\mu=-2\Delta,\quad \scal{\E_2,\E_4}_\mu=4\E_6\E_2+\mu\E_4^2,\quad \scal{\E_2,\E_6}_\mu=6\E_4^2\E_2-2\mu\E_4\E_6. %
\]
Moreover, %
\begin{enumerate}[i)]
\item  this Poisson bracket is Jacobian with potential %
\[%
k_\mu=-2\Delta\E_2+\mu\E_4^2\E_6.
\]%
\item The Poisson algebras \(\left(\fQM,\scal{\phantom{f},\phantom{g}}_\mu\right)\) and \(\left(\fQM,\scal{\phantom{f},\phantom{g}}_{\mu'}\right)\) are Poisson modular iso\-mor\-phic for all \(\mu\) and \(\mu'\) in \(\CC^*\). %
\item For any \(\mu\in\CC\), the Poisson center of \(\left(\fQM,\scal{\phantom{f},\phantom{g}}_\mu\right)\) is the polynomial algebra \(\CC[k_\mu]\). %
\end{enumerate}
\end{propintro}
\begin{remarkintro}
We note that \(\scal{\phantom{f},\phantom{g}}_0=\parenth{\phantom{f},\phantom{g}}_4\). %
\end{remarkintro}
Finally, the following result implies that our classification is complete. %
\begin{theoremintro}\label{thm_C}
Up to Poisson modular isomorphism, the only distinct admissible Poisson brackets on the algebra of quasimodular forms are \(\accol{\phantom{f},\phantom{g}}_1\), \(\scal{\phantom{f},\phantom{g}}_1\) and the family \(\parenth{\phantom{f},\phantom{g}}_\alpha\) for any \(\alpha\in\CC\). %
\end{theoremintro}
\begin{remarkintro}
We could endow the algebra of modular forms with another Poisson structure \(\bi\). If we require \(\bi\left(\Mod_k,\Mod_\ell\right)\subset\Mod_{k+\ell+2}\), then \(\bi\) is necessarily defined by \(\bi(\E_4,\E_6)=\alpha\E_4^3+\beta\E_6^2\) for some complex numbers \(\alpha\) and \(\beta\). If \(\alpha\beta\neq 0\) then \(\left(\Mod_*,\bi\right)\) is Poisson isomorphic to \(\left(\Mod_*,\RC_1\right)\) and is indeed studied by this work. If \(\alpha\beta=0\), the Poisson algebras are no more Poisson isomorphic (they do not have the same group of automorphisms). This degenerate case deserves another study. %
\end{remarkintro}

\begin{remarkintro}
From an algebraic point of view, classifications of Poisson structures and associated (co)homology for polynomial algebras in two variables appear in \cite{MR1910942}, \cite{PiThese} or \cite{MR1900294} for a Poisson bracket on \(\CC[x,y]\) defined by \(\accol{x,y}=\varphi(x,y)\) with \(\varphi\) an homogeneous or a square free weight homogeneous polynomial in \(\CC[x,y]\). The algebra of modular forms \(\fMod=\CC[E_4,E_6]\) with the Poisson bracket defined by \(\RC_1\) is the case \(A_2\) in the classification theorem 3.8 in \cite{MR1910942}. Applying propositions 4.10 and 4.11 of \cite{PiThese}, or theorems 4.6 and 4.11 of  \cite{MR1910942}, we can deduce that the Poisson cohomology spaces \(\HP^1(\fMod)\) and \(\HP^2(\fMod)\) are of respective dimensions 1 and 2. In three variables, the Poisson structures on the algebra \(\fQM=\CC[E_2,E_4,E_6]\) of quasimodular forms arising from theorem C above do not fall under the classification of \cite{MR1086519} since there are not quadratic. The (co)homological study of \cite{PiThese} and \cite{MR2228339} does not apply to the brackets \(\accol{\phantom{x},\phantom{y}}_1\) and \(\parenth{\phantom{x},\phantom{y}}_\alpha\) since they are not Jacobian, or to the Jacobian bracket \(\scal{\phantom{x},\phantom{y}}_1\) because its potential \(k_1\) does not admit an isolated singularity at the origin. 
\end{remarkintro}
%.............................................
\section{Number theoretic and algebraic background}
\subsection{Quasimodular forms}
The aim of this section is to provide the necessary background on quasimodular forms. For details, the reader is advised to refer to~\cite{MR2409678} or~\cite{MR2186573}. On \(\sldz\), a modular form of weight \(k\in2\N\), \(k\neq 2\), is a holomorphic function on the Poincar\'e upper half plane \(\pk=\{z\in\CC\colon \Im z>0\}\) satisfying %
\[%
(cz+d)^{-k}f\left(\frac{az+b}{cz+d}\right)=f(z) %
\]
for any \(\bigl(\begin{smallmatrix}a & b\\ c & d\end{smallmatrix}\bigr)\in\sldz\) and having Fourier expansion %
\[%
f(z)=\sum_{n\geq 0}\widehat{f}(n)\eexp^{2\pi\ic nz}. %
\]
We denote by \(\Mod_k\) the finite dimensional space of modular forms of weight \(k\). The algebra of modular forms is defined as the graded algebra %
\[%
\Mod_*=\bigoplus_{\substack{k\in2\N\\ k\neq 2}}\Mod_k. %
\]
Let \(k\geq 2\) be even. We define the Eisenstein series of weight \(k\) by %
\[%
\E_k(z)=1-\frac{2k}{B_k}\sum_{n=1}^{+\infty}\sigma_{k-1}(n)\eexp^{2\pi\ic nz}. %
\]
Here the rational numbers \(B_k\) are defined by their exponential generating series %
\[%
\sum_{n=0}^{+\infty}B_n\frac{t^n}{n!}=\frac{t}{\eexp^t-1}
\]
and \(\sigma_{k-1}\) is the divisor function defined by %
\[%
\sigma_{k-1}(n)=\sum_{\substack{d\mid n\\ d>0}}d^{k-1} \qquad(n\in\N^*). %
\]
If \(k\geq 4\), the Eisenstein series \(\E_k\) is a modular form of weight \(k\) and \(\Mod_*\) is the polynomial algebra in the two algebraically independent Eisentsein series \(\E_4\) and \(\E_6\). In other words, %
\[%
\Mod_*=\CC[\E_4,\E_6],\quad \Mod_k=\bigoplus_{\substack{(i,j)\in\N^2\\ 4i+6j=k}}\CC\E_4^i\E_6^j. %
\]
However, the Eisenstein series \(\E_2\) is not a modular form. It satisfies %
\[%
(cz+d)^{-2}\E_2\left(\frac{az+b}{cz+d}\right)=\E_2(z)+\frac{6}{\pi\ic}\frac{c}{cz+d}\qquad(z\in\pk) %
\]
for any \(\bigl(\begin{smallmatrix}a & b\\ c & d\end{smallmatrix}\bigr)\in\sldz\). Moreover, the algebra of modular forms is not stable by the normalised complex derivation %
\[%
\Diff=\frac{1}{2\pi\ic}\frac{\dd}{\dd z}. %
\]
For example, we have the following Ramanujan differential equations %
\begin{align*}
\Diff\E_2 &= \dfrac{1}{12}\left(\E_2^2-\E_4\right)\\
\shortintertext{}
\Diff\E_4 &= \dfrac{1}{3}\left(\E_4\E_2-\E_6\right)\\
\shortintertext{}
\Diff\E_6 &= \dfrac{1}{2}\left(\E_6\E_2-\E_4^2\right). %
\end{align*}
To account for these observations, and using the fact that \(\E_2\), \(\E_4\) and \(\E_6\) are algebraically independent, we introduce the algebra \(\fQM\) of quasimodular forms defined as the polynomial algebra %
\[%
\fQM=\CC[\E_2,\E_4,\E_6]=\Mod_*[\E_2]. %
\]
More intrinsically, if for \(\gamma=\bigl(\begin{smallmatrix}a & b\\ c & d\end{smallmatrix}\bigr)\in\sldz\) we define %
\[%
\X(\gamma)=z\mapsto\frac{c}{cz+d} %
\]
and %
\[%
f\vert_k\gamma=z\mapsto(cz+d)^{-k}f\left(\frac{az+b}{cz+d}\right) %
\]
then a quasimodular form of weight \(k\in 2\N\) and depth \(s\in\N\) is a holomorphic function \(f\) on \(\pk\) such that there exist holomorphic functions \(f_0,\dotsc,f_s\) (\(f_s\neq 0\)) satisfying %
\[%
f\vert_k\gamma=\sum_{j=0}^sf_j\X(\gamma)^j %
\]
for any \(\gamma\in\sldz\). Moreover, it is required that any \(f_j\) has a Fourier expansion %
\[%
f_j(z)=\sum_{n\geq 0}\widehat{f_j}(n)\eexp^{2\pi\ic nz}\qquad(z\in\pk). %
\]
The zero function is supposed to have arbitrary weight and depth \(0\). We write \(\Mod_k^{\leq\infty}\) for the space of quasimodular forms of weight \(k\) and \(\Mod_k^{\leq s}\) for the space of quasimodular forms of weight \(k\) and depth less than or equal to \(s\). We have \(\Mod_k^{\leq 0}=\Mod_k\) and %
\[%
\Mod_k^{\leq s}=\bigoplus_{j=0}^s\Mod_{k-2j}\E_2^j,\qquad\fQM=\bigoplus_{k\in 2\N}\Mod_k^{\leq k/2}. %
\]
Moreover \(\Diff\Mod_k^{\leq s}\subset\Mod_{k+2}^{\leq s+1}\). Since the depth of a quasimodular form is nothing but its degree as a polynomial in \(\E_2\) with modular coefficients we note that %
\[%
\Mod_k^{\leq s}=\bigoplus_{j=0}^s\Mod_{k}^t,\qquad\fQM=\bigoplus_{k\in 2\N}\bigoplus_{t=0}^{k/2}\Mod_{k}^t %
\]
where %
\[%
\Mod_{k}^t=\Mod_{k-2t}\E_2^t=\bigoplus_{\substack{(i,j)\in\N^2\\ 4i+6j=k-2t}}\CC\E_4^i\E_6^j\E_2^t . %
\]
An important element in our study will be the discriminant function \(\Delta=\E_4^3-\E_6^2\). We note that \(\Diff\Delta=\Delta\E_2\). %

Let \(n\) be a non negative integer, \(f\) a modular form of weight \(k\) and \(g\) a modular form of weight \(\ell\). The \(n\)-th Rankin-Cohen bracket of \(f\) and \(g\) is %
\[%
\RC_n(f,g)=\sum_{r=0}^n(-1)^r\binom{k+n-1}{n-r}\binom{\ell+n-1}{r}\Diff^r f\Diff^{n-r}g.
\]
This is a modular form of weight \(k+\ell+2n\). If \(f\) and \(g\) are quasimodular forms of respective weights \(k\) and \(\ell\) and respective depths \(s\) and \(t\), their \(n\)-th Rankin-Cohen bracket is defined in~\cite{MR2568053} by %
\begin{equation}\label{def_MR}%
\RC_n(f,g)=\sum_{r=0}^n(-1)^r\binom{k-s+n-1}{n-r}\binom{\ell-t+n-1}{r}\Diff^r f\Diff^{n-r}g.
\end{equation}
This is a quasimodular form of weight \(k+\ell+2n\) and minimal depth (that is \(s+t\)).
%.............................................
\subsection{Poisson algebra}
The aim of this section is to give a brief account on what is needed about Poisson algebra. For more details, the reader is advised to refer to~\cite{MR2906391}.  A commutative \(\CC\)-algebra \(A\) is a Poisson algebra if there exists a bilinear skew-symmetric map \(\bi\colon A\times A\to A\) satisfying the two conditions: %
\begin{itemize}
\item Leibniz rule \(\bi(fg,h)=f\bi(g,h)+\bi(f,h)g\)
\item Jacobi identity \(\bi\left(f,\bi(g,h)\right)+\bi\left(g,\bi(h,f)\right)+\bi\left(h,\bi(f,g)\right)=0\) %
\end{itemize}
for all \(f\), \(g\) and \(h\) in \(A\). The bilinear map \(\bi\) is given the name of Poisson bracket. If \(A\) is a finitely generated algebra with generators \(x_1,\dotsc,x_N\), a Poisson bracket \(\bi\) is entirely determined by its values \(\bi(x_i,x_j)\) for \(i<j\) where \(A\) is generated by \(x_1,\dotsc,x_N\). More precisely, we have %
\begin{equation}\label{eq_jacob}%
\bi(f,g)=\sum_{0\leq i<j\leq N}\left(\frac{\partial f}{\partial x_i}\frac{\partial g}{\partial x_j}-\frac{\partial g}{\partial x_i}\frac{\partial f}{\partial x_j}\right)\bi(x_i,x_j) %
\end{equation}
for \(f\) and \(g\) expressed as polynomials in \(x_1,\dotsc,x_N\). %

If \(A=\CC[x,y]\), any \(p\in A\) determines a Poisson bracket satisfying \(\bi(x,y)=p\). However, if \(A=\CC[x,y,z]\), for any \(p\), \(q\) and \(r\) in \(A\), there exists a Poisson bracket on \(A\) defined by %
\[%
\bi(x,y)=r,\quad \bi(y,z)=p,\quad\text{and}\quad\bi(z,x)=q %
\]
if and only if %
\begin{equation}\label{eq_curlcondition}
\curl(p,q,r)\cdot(p,q,r)=0 %
\end{equation}
where %
\[%
\curl(p,q,r)=\left(\frac{\partial r}{\partial y}-\frac{\partial q}{\partial z},\frac{\partial p}{\partial z}-\frac{\partial r}{\partial x},\frac{\partial q}{\partial x}-\frac{\partial p}{\partial y}\right). %
\]
If condition~\eqref{eq_curlcondition} is satisfied, then \((p,q,r)\) is called a Poissonian triple.  A particular case is obtained if there exists \(k\in\CC[x,y,z]\) such that \(\curl(p,q,r)=(p,q,r)\wedge\grad k\). The bracket \(\bi\) is said then to be unimodular. Among unimodular brackets are the Jacobian brackets. A bracket \(\bi\) is Jacobian if \((p,q,r)=\grad k\) for some polynomial \(k\). The bracket \(\bi\) satisfies then %
\[%
\bi(f,g)=\jac(f,g,k)\qquad (f,g\in A)
\]
In this case, \(\CC[x,y,z]\) is said to have a Jacobian Poisson structure (JPS) of potential (or Casimir function) \(k\). The Poissonian triple  \((p,q,r)\)  is said then to be exact. %

The Poisson center (or zeroth Poisson cohomology group) of a Poisson algebra \(A\) is the Poisson subalgebra %
\[%
\HP^0(A)=\left\{g\in A\colon\bi(f,g)=0,\, \forall f\in A\right\}. 
\]
The Poisson center is contained in the Poisson centraliser of any element in the algebra: let \(f\in A\), its Poisson centraliser is %
\(%
\{g\in A\colon\bi(f,g)=0\}.
\)
The following lemma computes the Poisson center of polynomial algebras in three variables equipped with a Jacobian Poisson structure. It allows one to recover for example \cite[Poposition 4.2]{MR2228339} in the particular case where the potential is a weight homogeneous polynomial with an isolated singularity. A polynomial \(h\in\CC[x,y,z]\) is indecomposable if there is no polynomial \(p\in\CC[x]\) with \(\deg p\geq 2\) such that \(h=p\circ\ell\) for some \(\ell\in\CC[x,y,z]\). %
\begin{lemma}\label{lem_lemma0}
Let \(\CC[x,y,z]\) be endowed with a Jacobian Poisson structure of non constant potential \(k\). Its Poisson center is \(\CC[k]\) if and only if \(k\) is indecomposable. %
\end{lemma}
\begin{proof}
Assume that \(k\) is not indecomposable: \(k=p\circ\ell\) with \(p\in\CC[x]\), \(\deg p=2\). Then \(\jac(\ell,g,k)=(p'\circ\ell)\jac(\ell,g,\ell)\) hence \(\ell\) is in the Poisson center however not in \(\CC[k]\). Assume conversely that \(k\) is indecomposable. Let \(f\) be in the Poisson center, then the rank of the Jacobian matrix of \((f,g,k)\) is at most \(2\) for any \(g\). If it is \(1\) for any \(g\) then \(\grad f\) and \(\grad k\) are zero, it contradicts the fact that \(k\) is not constant. Hence, for some \(g\), the rank is \(2\). It follows (see e.g.~\cite[Theorem 6]{MR2267134}) that there exists \(q\in\CC[x,y,z]\), \(F\in\CC[x]\) and \(K\in\CC[x]\) such that \(f=F\circ q\) and \(k=K\circ q\). Since \(k\) is indecomposable and non constant, we have \(\deg K=1\) hence \(q\), and \(f\), are polynomials in \(k\). %
\end{proof}

If \(A\) and \(B\) are two Poisson algebras, with respective Poisson brackets \(\bi_A\) and \(\bi_B\), a map \(\varphi\colon A\to B\) is a morphism of Poisson algebras when it is a morphism of algebras that satisfies %
\[%
\varphi\left(\bi_A(f,g)\right)=\bi_B\left(\varphi(f),\varphi(g)\right) %
\]
for any \(f\) and \(g\) in \(A\). Two Poisson-isomorphic Poisson algebras have isomorphic Poisson centres. %

We detail now a canonical way to extend a Poisson structure from an algebra \(A\) to a polynomial algebra \(A[x]\). This construction is due to Sei-Qwon Oh~\cite{MR2220812}. A Poisson derivation of \(A\) is a derivation \(\sigma\) of \(A\) satisfying %
\[%
\sigma\left(\bi(f,g)\right)=\bi\left(\sigma(f),g\right)+\bi\left(f,\sigma(g)\right) %
\]
for all \(f\) and \(g\) in \(A\). If \(\sigma\) is a Poisson derivation of \(A\), a Poisson \(\sigma\)-derivation is a derivation \(\delta\) of \(A\) such that %
\[%
\delta\left(\bi(f,g)\right)=\bi\left(\delta(f),g\right)+\bi\left(f,\delta(g)\right)+\sigma(f)\delta(g)-\delta(f)\sigma(g) %
\]
for all \(f\) and \(g\) in \(A\). %
\begin{theorem}[Oh]\label{thm_un}
Let \((A,\bi_A)\) be a Poisson algebra. Let \(\sigma\) and \(\delta\) be linear maps on \(A\). The polynomial ring \(A[x]\) becomes a Poisson algebra with Poisson brackets \(\bi\) defined by %
\begin{align*}
\bi(f,g) &= \bi_A(f,g)\\
\bi(x,f) &= \sigma(f)x+\delta(f) %
\end{align*}
for all \(f\) and \(g\) in \(A\) if and only if \(\sigma\) is a Poisson derivation and \(\delta\) is a Poisson \(\sigma\)-derivation. In this case, the Poisson algebra \(A[x]\) is said to be a Poisson-Ore extension of \(A\). It is denoted \(A[x]_{\sigma,\delta}\). %
\end{theorem}

We describe also a general process to obtain Poisson brackets from a pair of derivations. A pair \((\delta,\der)\) of two derivations of \(A\) is solvable if there exists some scalar \(\alpha\) such that \(\delta\circ\der-\der\circ\delta=\alpha\der\). In particular, a solvable pair \((\delta,\der)\) is abelian when \(\alpha=0\). %
\begin{proposition}\label{prop_generviasolvable}
Let \(A\) be a commutative algebra, \(\der\) and \(\delta\) two derivations of \(A\). Let \(\bi\colon A\times A\to A\) be defined by %
\[%
\bi(f,g)=\delta(f)\der(g)-\der(f)\delta(g)\qquad(f,g\in A).
\]
Then %
\begin{enumerate}[i)]
\item\label{pointi} the map \(\bi\) is bilinear skew-symmetric and satisfies Leibniz rule.
\item\label{previousassumption} If \((\delta,\der)\) is solvable, then \(\bi\) satisfies Jacobi identity and so becomes a Poisson bracket.
\item\label{pointiii} Under the hypothesis in~\ref{previousassumption}), \(\der\) is a Poisson derivation for \(\bi\). % 
\end{enumerate}
\end{proposition}
\begin{proof}[\proofname{} of Proposition~\ref{prop_generviasolvable}]
Point~\ref{pointi}) is immediate. Point~\ref{previousassumption}) is a consequence of the following computation. If \((\delta,\der)\) is solvable with \(\delta\der-\der\delta=\alpha\der\) and if \(\Bi\colon A\otimes A\otimes A\otimes A\to A\) is defined by \(\Bi(f,g,h)=\bi\left(f,\bi(g,h)\right)\) then %
\begin{align*}
\Bi &= \alpha(\der\otimes\der\otimes\delta-\der\otimes\delta\otimes\der)\\
&\phantom{=}+\left(\delta\otimes(\der\circ\delta)\otimes\der-\der\otimes\delta\otimes(\der\circ\delta)\right)\\
&\phantom{=}+\left(\der\otimes(\der\circ\delta)\otimes\delta-\delta\otimes\der\otimes(\der\circ\delta)\right)\\
&\phantom{=}+(\delta\otimes\delta\otimes\der^2-\delta\otimes\der^2\otimes\delta)\\
&\phantom{=}+(\der\otimes\der\otimes\delta^2-\der\otimes\delta^2\otimes\der). %
\end{align*}
Point~\ref{pointiii}) is obtained by direct computation. %
\end{proof}
A direct consequence of this Proposition is the following Corollary. If \(A=\bigoplus_{n\geq 0}A_n\) is a commutative graded algebra, a map \(\kappa\colon A\to\CC\) is graded-additive if, for any \(k\) and \(\ell\), for any \(f\in A_k\) and any \(g\in A_\ell\), we have \(\kappa(fg)=\kappa(f)+\kappa(g)\). %%
\begin{corollary}\label{cor_trois}
Let \(A=\bigoplus_{n\geq 0}A_n\) be a commutative graded algebra. Let \(\kappa\colon A\to\CC\) be a graded-additive map. Let \(\der\) be a homogeneous derivation of \(A\) (there exists \(e\geq 0\) such that \(\der A_n\subset A_{n+e}\) for any \(e\geq 0\)). Then, the bracket defined on \(A\) by the bilinear extension of %
\[%
\bi(f,g)=\kappa(f)f\der(g)-\kappa(g)g\der(f)\qquad(f\in A_k,\, g\in A_\ell) %
\] 
is a Poisson bracket for which \(\der\) is a Poisson derivation. %
\end{corollary}

We turn on formal deformations of a commutative \(\CC\)-algebra \(A\). Assume we have a family \(\boldmu=(\mu_i)_{i\in\N}\) of bilinear maps \(\mu_i\colon A\times A\to A\) such that \(\mu_0\) is the product. Let \(A[[\hslash]]\) be the commutative algebra of formal power series in one variable \(\hslash\) with coefficients in \(A\). The family \(\boldmu\) is a formal deformation of \(A\) if the non commutative product on \(A[[\hslash]]\) defined by extension of %
\[%
f\ast g=\sum_{j\geq 0}\mu_j(f,g)\hslash^j\qquad(f,g\in A) %
\]
is associative. This condition is equivalent to %
\begin{equation}\label{associativity}
\sum_{r=0}^n\mu_{n-r}(\mu_r(f,g),h)=\sum_{r=0}^n\mu_{n-r}(f,\mu_r(g,h))\qquad{\text{ for all }} f,g,h\in A %
\end{equation}
for all \(n\geq 0\). In this case, the product \(\ast\) is called a star product. If \(\boldmu\) is a formal deformation and if moreover \(\mu_1\) is skew-symmetric and \(\mu_2\) is symmetric, then \(\left(A,\mu_1\right)\) is a Poisson algebra. %
%................................
\subsection{Problems at issue}
The first Rankin-Cohen bracket %
\[%
\RC_1(f,g)=kf\Diff(g)-\Diff(f)\ell g\qquad(f\in\Mod_k,\, g\in\Mod_\ell) %
\]
gives \(\Mod_*\) a structure of Poisson algebra. This is a consequence of Corollary~\ref{cor_trois}. Cohen, Manin \& Zagier~\cite{MR1418868} and Yao~\cite{Yao_PHD} or~\cite{MR2838685} proved that the family of Rankin-Cohen brackets is a formal deformation of \(\Mod_*\). In this case, the star product is called the Eholzer product. This subject has been widely studied then. See for example~\cite{MR1783553,MR2443940}. %

Can we construct formal deformations of \(\fQM\)? In other words, can we construct suitable families \((\mu_n)_{n\in\N}\) of bilinear maps on \(\fQM\) that increase the weight by \(2n\), preserve the depth and define an analogue of Eholzer product? The brackets defined in~\eqref{def_MR} do not lead to a solution since \(\RC_1\) does not even provide \(\fQM\) with a Poisson structure. Our first step is to obtain admissible Poisson brackets on \(\fQM\) with the following definition. %
\begin{definition}\label{dfn_quatre}
A Poisson bracket \(\bi\) on \(\fQM\) is admissible if %
\begin{enumerate}[1)]
\item\label{item_admi} \(\bi(f,g)=\RC_1(f,g)\) if \(f\) and \(g\) are in \(\Mod_*\)
\item\label{item_admii}  it satisfies \(\bi\left(\Mod_k^{\leq s},\Mod_\ell^{\leq t}\right)\subset\Mod_{k+\ell+2}^{\leq s+t}\) for all \(k,\ell,s,t\).%
\end{enumerate}
\end{definition}
\begin{remark}
We could have replace condition~\ref{item_admii} by the following one: there exists \(e\geq 0\) such that \(\bi\left(\Mod_k^{\leq s},\Mod_\ell^{\leq t}\right)\subset\Mod_{k+\ell+e}^{\leq s+t}\) for all \(k,\ell,s,t\). However, condition~\ref{item_admi} implies that necessarily \(e=2\). %
\end{remark}
Equivalently, a Poisson bracket \(\bi\) on \(\fQM\) is admissible if and only if %
\begin{align*}
\bi(\E_4,\E_6)&=-2\Delta\\
\bi(\E_2,\E_4)&\in\Mod_8^{\leq\infty},\quad \bi(\E_2,\E_6)\in\Mod_{10}^{\leq\infty}\\
\bi(\E_2,\Mod_*)&\subset\Mod_*\E_2+\Mod_*. %
\end{align*}
In order to classify the admissible Poisson brackets, we introduce the notion of Poisson modular isomorphism. %
\begin{definition}
A Poisson isomorphism \(\varphi\colon(\fQM,\bi_1)\to(\fQM,\bi_2)\) is called a Poisson modular isomorphism if \(\varphi(\Mod_*)\subset \Mod_*\). %
\end{definition}
Indeed, if \(\varphi\) is a Poisson modular isomorphism then, its restriction to the subalgebra \(\Mod_*\) is the identity. This is a consequence of the following Proposition. %
\begin{proposition}\label{prop_six}
The group of Poisson automorphisms of the Poisson algebra \((\Mod_*,\RC_1)\) is trivial. %
\end{proposition}
\begin{proof}
Let \(\varphi\) be a Poisson automorphism of \(\Mod_*\). There exist two polynomials \(s\) and \(t\) in \(\CC[x,y]\) such that \(\varphi(\E_4)=s(\E_4,\E_6)\) and \(\varphi(\E_6)=t(\E_4,\E_6)\). By~\eqref{eq_jacob} we have \[\RC_1\left(\varphi(\E_4),\varphi(\E_6)\right)=\jac(s,t)(\E_4,\E_6)\cdot\RC_1(\E_4,\E_6).\] Since \(\varphi\) is an automorphism, \(\jac(s,t)\) is a non zero scalar, say \(\lambda\). We get %
\begin{equation}\label{eq_tsJe}%
\varphi\left(\RC_1(\E_4,\E_6)\right)=\lambda\RC_1(\E_4,\E_6) %
\quad\text{hence}\quad %
s^3-t^2=\lambda(x^3-y^2)\quad\text{in \(\CC[x,y]\).} %
\end{equation}
We develop \(s\) and \(t\) into homogeneous components with respect to the weight: %
\[%
s=\sum_{\substack{i=0\\ i\neq 1}}^ms_{2i}\quad\text{and}\quad t=\sum_{\substack{i=0\\ i\neq 1}}^nt_{2i}
\]
where %
\[%
s_{2i}=\sum_{\substack{(a,b)\in\N^2\\ 2a+3b=i}}\sigma_{a,b}x^ay^b\quad\text{and}\quad t_{2i}=\sum_{\substack{(a,b)\in\N^2\\ 2a+3b=i}}\tau_{a,b}x^ay^b\qquad(\sigma_{a,b},\,\tau_{a,b}\in\CC)
\]
for all i (where \(m=0\) or \(m\geq 2\) and \(n=0\) or \(n\geq 2\)).  Equation \eqref{eq_tsJe} implies that  \(t(E_4,E_6)^2-s(E_4,E_6)^3\) has weight \(12\). Then only three cases are possible. 
\begin{enumerate}[1)]
\item If \(3m>2n\) then \(m=2\) and so \(n\in\{0,2\}\). It implies that \(s=\sigma_{00}+\sigma_{10} x\) and \(t=\tau_{00}+\tau_{10} x\). This contradicts \(\jac(s,t)\neq 0\). %
\item If \(3m<2n\) then \(n=3\) and \(m=0\). This contradicts \(\jac(s,t)\neq 0\). %
\item If \(3m=2n\). We differentiate \eqref{eq_tsJe} with respect to \(x\) and \(y\) and get 
\[3s^2\dfrac{\partial s}{\partial x}-2t\dfrac{\partial t}{\partial x} = 3\lambda x^2,\qquad 3s^2\dfrac{\partial s}{\partial y}-2t\dfrac{\partial t}{\partial y} = -2\lambda y.\]
\end{enumerate}
This implies 
\begin{equation}\label{PDE}
2t = 3x^2\dfrac{\partial s}{\partial y}+2y\dfrac{\partial s}{\partial x},\qquad 
3s^2 = 3x^2\dfrac{\partial t}{\partial y}+2y\dfrac{\partial t}{\partial x}.
\end{equation}
From the first differential equation of \eqref{PDE} we have %
\[%
2t(E_4,E_6)=3E_4^2\dfrac{\partial s}{\partial y}(E_4,E_6)+2E_6\dfrac{\partial s}{\partial x}(E_4,E_6). %
\]
The highest weight of the right hand side is less than or equal to \(2m+2\). This implies \(n\leq m+1\). From the second differential equation of \eqref{PDE} we have %
\[%
3s^2(E_4,E_6) = 3E_4^2\dfrac{\partial t}{\partial y}(E_4,E_6)+2E_6\dfrac{\partial t}{\partial x}(E_4,E_6)%
\]
hence \(2m\leq n+1\). We deduce \((m,n)\in\{(0,0),(2,3)\}\). Since \(n=m=0\) would imply \(\jac(s,t)=0\) we have \(n=3\) and \(m=2\). Then \(s=\sigma_{00}+\sigma_{10}x\) and \(t=\tau_{00}+\tau_{10}x+\tau_{01}y\). The first differential equation in \eqref{PDE} implies that \(\tau_{00}=\tau_{10}=0\) and \(\tau_{01}=\sigma_{10}\) whereas the second one implies that \(\sigma_{00}=0\) and \(\sigma_{10}=1\). Finally \(\varphi(E_4)=s(E_4,E_6)=E_4\) and \(\varphi(E_6)=s(E_4,E_6)=E_6\). 
\end{proof}
Since  \(\RC_1(\Delta,f)=\left(12\Diff(f)-kf\E_2\right)\Delta\) for any \(f\in\Mod_k\), the first Rankin-Cohen bracket defines a derivation on \(\Mod_*\) called the Serre's derivative by linear extension of:
\begin{equation}\label{eq_defSerre}%
\vartheta f=\frac{\RC_1(\Delta,f)}{12\Delta}=\Diff(f)-\frac{k}{12}f\E_2\qquad(f\in\Mod_k). %
\end{equation}
This derivation is characterised by its values on the generators %
\[%
\vartheta\E_4=-\frac{1}{3}\E_6,\qquad \vartheta\E_6=-\frac{1}{2}\E_4^2. %
\]
We shall need the following result. %
\begin{proposition}\label{prop_sept}
The kernel of Serre's derivative is the Poisson centraliser of \(\Delta\) for the first Rankin-Cohen bracket. This is \(\CC[\Delta]\). %
\end{proposition}
\begin{proof}
If \(f\in\Mod_k\) is in \(\ker\vartheta\) then \(kf\Diff(\Delta)=12\Delta\Diff(f)\). Solving the differential equation we find that \(12\) divides \(k\) and that \(f\in\CC\Delta^{k/12}\). %
\end{proof}
We note that, for any \(g\in\Mod_\ell\), we have %
\[%
\RC_1(\Delta^m,g)=m\Delta^m\left(12\Diff(g)-\ell g\E_2\right)
\]
and deduce that, for any \(f\in\CC[\Delta]\) and \(g\in\Mod_*\) we have %
\begin{equation}\label{eq_RCEuler}%
\RC_1(f,g)=12\xi(f)\vartheta(g)
\end{equation}
where \(\xi\) is the Eulerian derivative on \(\CC[\Delta]\) defined by \(\xi=\Delta\dfrac{\partial}{\partial\Delta}\). %
%.............................................
\section{Poisson structures on quasimodular forms}
\subsection{First family}
This section is devoted to the proof of Proposition~\ref{prop_A}. %

We fix \(\lambda\in\CC^*\) and introduce in \(\CC[x,y,z]\) the three polynomials %
\begin{align*}
r(x,y,z) &= \frac{1}{3}(\lambda y^2-2xz)\\
p(x,y,z) &=-2(y^3-z^2)\\
q(x,y,z) &=-\frac{1}{2}(\lambda yz-2xy^2). %
\end{align*}
Since \((p,q,r)\cdot\curl(p,q,r)=0\), we define a Poisson bracket on \(\fQM\) if we set %
\begin{align*}%
\accol{\E_4,\E_6}_\lambda &=p(\E_2,\E_4,\E_6)\\
\accol{\E_2,\E_4}_\lambda &=r(\E_2,\E_4,\E_6)\\
\accol{\E_6,\E_2}_\lambda &=q(\E_2,\E_4,\E_6). %
\end{align*}
Let us prove that \(\accol{\phantom{f},\phantom{f}}_\lambda\) is not unimodular. If it were, we would have \(k\in\CC[x,y,z]\) such that \(\curl(p,q,r)=(p,q,r)\wedge\grad k\). Identifying the first components would lead to %
\[%
\frac{7}{6}\lambda y=\frac{1}{2}(-\lambda yz+2y^2x)\frac{\partial k}{\partial z}-\frac{1}{3}(\lambda y^2-2zx)\frac{\partial k}{\partial y}
\]
that has no solution in \(\CC[x,y,z]\). %

A Poisson modular isomorphism \(\varphi_\lambda\) between \(\left(\fQM,\accol{\phantom{f},\phantom{f}}_\lambda\right)\) and  \(\left(\fQM,\accol{\phantom{f},\phantom{f}}_1\right)\) is determined by %
\[%
\varphi_\lambda(\E_2)=\lambda\E_2,\quad\varphi_\lambda(\E_4)=\E_4,\quad\varphi_\lambda(\E_6)=\E_6. %
\]

Finally, we determine the Poisson center of the Poisson algebra \(\left(\fQM,\accol{\phantom{f},\phantom{f}}_1\right)\). Let us define a derivation on \(\Mod_*\) by %
\(\sigma=2\vartheta\) (see~\eqref{eq_defSerre}) an a derivation on \(\Mod_*\) by linear extension of %
\[%
\delta(f)=\frac{k}{12}f\E_4\qquad (f\in\Mod_k).
\]
We note that \(\left(\fQM,\accol{\phantom{f},\phantom{f}}_1\right)\) is the Poisson-Ore extension \(\CC[\E_4,\E_6][\E_2]_{\sigma,\delta}\)
Now, consider any \(f\in\fQM\) written as %
\[%
f=\sum_{i=0}^sf_i\E_2^i,\qquad f_i\in\Mod_*. %
\]
We compute %
\[%
\accol{\E_2,f}_1=\delta(f_0)+\sum_{i=1}^s\left(\sigma(f_{i-1})+\delta(f_i)\right)\E_2^i+\sigma(f_s)\E_2^{s+1}. %
\]
If \(\accol{\E_2,f}_1=0\) then \(\delta(f_0)=0\) hence \(f_0\in\CC\) and \(\sigma(f_0)=0\). We obtain inductively that \(f_i\in\CC\) for all \(0\leq i\leq s\) so that the Poisson centraliser of \(\E_2\) is \(\CC[\E_2]\). Suppose that the Poisson center contains a non scalar element. Then it is in the Poisson centraliser of \(\E_2\) and can be written %
\[%
f=\sum_{j=0}^p\alpha_j\E_2^j,\qquad p\geq 1,\, \alpha_j\in\CC,\,\alpha_p\neq 0. %
\]
We compute %
\[%
\accol{\E_4,f}_1=\sum_{j=0}^pj\alpha_j\E_2^{j-1}\cdot\accol{\E_4,\E_2}_1 %
\]
and find that the coefficient of \(\E_2^p\) is non zero. It follows that \(f\) is not in the Poisson center. %
%.............................;
\subsection{Second family}\label{sec_troisdeux}
This section is devoted to the proof of Proposition~\ref{prop_B}. %
We fix \(\alpha\in\CC\) and introduce in \(\CC[x,y,z]\) the three polynomials %
\begin{align*}
r(x,y,z) &= \alpha xz\\
p(x,y,z) &=-2(y^3-z^2)\\
q(x,y,z) &=-\frac{3}{2}\alpha xy^2. %
\end{align*}
Since \((p,q,r)\cdot\curl(p,q,r)=0\), we define a Poisson bracket on \(\fQM\) if we set %
\begin{align*}%
\parenth{\E_4,\E_6}_\alpha &=p(\E_2,\E_4,\E_6)\\
\parenth{\E_2,\E_4}_\alpha &=r(\E_2,\E_4,\E_6)\\
\parenth{\E_6,\E_2}_\alpha &=q(\E_2,\E_4,\E_6). %
\end{align*}

Assume \(\alpha\neq 4\). Let us prove that \(\parenth{\phantom{f},\phantom{f}}_\alpha\) is not unimodular. If it were, we would have \(k\in\CC[x,y,z]\) such that \(\curl(p,q,r)=(p,q,r)\wedge\grad k\). Identifying the second components would lead to %
\[%
(4-\alpha)z=\alpha xz\frac{\partial k}{\partial x}+2(y^3-z^2)\frac{\partial k}{\partial z}
\]
that has no solution in \(\CC[x,y,z]\). %

If \(\alpha=4\), then \((p,q,r)=\grad k_0\) where \(k_0=-2(y^3-z^2)x\). As a consequence, the bracket \(\parenth{\phantom{f},\phantom{f}}_4\) provides \(\fQM\) with a Jacobian Poisson structure of potential \(k_0=-2\Delta\E_2=-2\Diff(\Delta)\). 

If \(\varphi\colon\left(\fQM,\parenth{\phantom{f},\phantom{f}}_\alpha\right)\to\left(\fQM,\parenth{\phantom{f},\phantom{f}}_{\alpha'}\right)\) is a Poisson modular isomorphism, let us prove that \(\alpha=\alpha'\). By Proposition~\ref{prop_six}, we have \(\varphi(\E_4)=\E_4\) and \(\varphi(\E_6)=\E_6\). By surjectivity, it follows that \(\varphi(\E_2)=\eta\E_2+F\) for some \(\eta\in\CC^*\) and \(F\in\Mod_*\). We compute %
\begin{align*}
\varphi\left(\parenth{\E_2,\E_4}_\alpha\right)&=\alpha\eta\E_6\E_2+\alpha\E_6F\\
\shortintertext{and}
\parenth{\varphi(\E_2),\varphi(\E_4)}_{\alpha'}&=\alpha'\eta\E_6\E_2+\parenth{F,\E_4}_{\alpha'}. %
\end{align*}
Since \(\parenth{F,\E_4}_{\alpha'}\in\Mod_*\) we get \(\alpha'=\alpha\). %

Finally, we determine the Poisson center of the Poisson algebra \(\left(\fQM,\parenth{\phantom{f},\phantom{f}}_\alpha\right)\). We note that \(\left(\fQM,\parenth{\phantom{f},\phantom{f}}_\alpha\right)\) is the Poisson-Ore extension \(\CC[\E_4,\E_6][\E_2]_{\sigma,\delta}\) where \(\sigma=-3\alpha\vartheta\) (see~\eqref{eq_defSerre}) and \(\delta=0\).
Let %
\[%
f=\sum_{j=0}^sf_j\E_2^j,\qquad f_j\in\Mod_*. %
\]
We have %
\[%
\parenth{\E_2,f}_\alpha=\sum_{j=0}^s\sigma(f_j)\E_2^{j+1} %
\]
hence \(f\) is in the Poisson centraliser of \(\E_2\) if and only each \(f_j\) is in the Poisson centraliser of \(\Delta\) for \(\RC_1\). By Proposition~\ref{prop_sept} we deduce that the centraliser of \(\E_2\) is \(\CC[\Delta,\E_2]\). Let %
\[
f=\sum_{j=0}^sf_j(\Delta)\E_2^j\in\CC[\Delta,\E_2]. %
\]
We use~\eqref{eq_RCEuler} to compute %
\[%
\parenth{f,\E_4}_\alpha=\sum_{j=0}^s\left(-4\xi(f_j)+j\alpha f_j\right)\E_6\E_2^j %
\]
and
\[%
\parenth{f,\E_6}_\alpha=\frac{3}{2}\sum_{j=0}^s\left(-4\xi(f_j)+j\alpha f_j\right)\E_4^2\E_2^j.  %
\]
We deduce that \(f\) is in the Poisson center of \(\parenth{\phantom{f},\phantom{f}}_\alpha\) if and only if %
\[%
\xi(f_j)=\frac{j\alpha}{4}f_j %
\]
for all \(j\), i.e. if and only if any \(f_j\) is of the form \(f_j=\lambda_j\Delta^{m_j}\) for some \(\lambda_j\in\CC\) and \(m_j\in\N\) such that \(j\alpha=4m_j\). If \(\alpha\notin\Q\) or if \(\alpha<0\) then \(j=0\) and \(m_j=0\) hence \(f=f_0\in\CC\). If \(\alpha=p/q\) with \(p\geq 1\), \(q\geq 1\) and \((p,q)=1\) then \(\lambda\Delta^{m_j}\E_2^j\) is in the Poisson center if and only if \(pj=4qm_j\). The result follows by obvious arithmetical consideration. Finally, if \(\alpha=0\), then \(\parenth{\phantom{f},\phantom{f}}_0\) is the trivial bracket and its Poisson center is \(\CC[\E_2]\). %
%...............................;
\subsection{Third family}
In this section, we study the third family, i.e. we prove Proposition~\ref{prop_C}. %

For any \(\mu\in\CC\), let us introduce %
\[%
k_\mu=-2\Delta\E_2+\mu\E_4^2\E_6. %
\]
Then %
\begin{align*}
\jac(\E_4,\E_6,k_\mu) &= \frac{\partial\mu_k}{\partial\E_2}=-2\E_4^3+\E_6^2\\
\jac(\E_2,\E_4,k_\mu) &= \frac{\partial\mu_k}{\partial\E_6}=4\E_6\E_2+\mu\E_4^3\\
\jac(\E_2,\E_6,k_\mu) &= -\frac{\partial\mu_k}{\partial\E_4}=6\E_4^2\E_2-2\mu\E_4\E_6. %
\end{align*}
The third family of Poisson bracket is then defined by \(\scal{f,g}_\mu=\jac(f,g,k_\mu)\). With the notation of Proposition~\ref{prop_B}, we have in particular \(\scal{f,g}_0=\parenth{f,g}_4\).%

For any \(\mu\in\CC^*\), we define a Poisson modular isomorphism \(\varphi_\mu\) between \(\left(\fQM,\scal{\phantom{f},\phantom{f}}_\mu\right)\) and \(\left(\fQM,\scal{\phantom{f},\phantom{f}}_1\right)\) by setting \(\varphi_\mu(\E_2)=\mu\E_2\), \(\varphi_\mu(\E_4)=\E_4\) and \(\varphi_\mu(\E_6)=\E_6\). %

Since the degree in \(\E_2\) of \(k_\mu\) as a polynomial in \(\E_2,\E_4,\E_6\) is \(1\), Lemma~\ref{lem_lemma0} implies that the Poisson center of \(\left(\fQM,\scal{\phantom{f},\phantom{f}}_\mu\right)\) is \(\CC[k_\mu]\). %
%...............................;
\subsection{Classification}
This section is devoted to the proof of Theorem~\ref{thm_C}. %

Let \(\accol{\phantom{f},\phantom{f}}\) be an admissible bracket on \(\fQM\). By Definition~\ref{dfn_quatre} and Theorem~\ref{thm_un} , there exist a Poisson derivation \(\sigma\) of \(\Mod_*\) and a Poisson \(\sigma\)-derivation \(\delta\) of \(\Mod_*\)  such that %
\[%
\accol{\E_2,f}=\sigma(f)\E_2+\delta(f)\qquad(f\in\Mod_*). %
\]
By definition, \(\sigma(\Mod_k)\subset\Mod_{k+2}\) and \(\delta(\Mod_k)\subset\Mod_{k+4}\) for any \(k\). The admissible bracket \(\accol{\phantom{f},\phantom{f}}\) is then defined by the four scalars \(\alpha\), \(\beta\), \(\gamma\) and \(\epsilon\) such that %
\[%
\sigma(\E_4)=\alpha\E_6,\, \delta(\E_4)=\beta\E_4^2,\, \sigma(\E_6)=\gamma\E_4^2,\text{and }\delta(\E_6)=\epsilon\E_4\E_6. %
\]
The condition that \(\sigma\) is a Poisson derivation imposes the condition %
\[%
\accol{\sigma(\E_4),\E_6}+\accol{\E_4,\sigma(\E_6)}=-2\sigma(\E_4^3-\E_6^2) %
\]
or equivalently \(3\alpha=2\gamma\). The condition that \(\delta\) is a Poisson \(\sigma\)-derivation imposes %
\[%
\delta\left(\accol{\E_4,\E_6}\right)=(2\beta+\epsilon)\E_4\accol{\E_4,\E_6}+\alpha\epsilon\E_4\E_6^2-\beta\gamma\E_4^4 %
\]
or equivalently %
\[%
\left\{%
\begin{aligned}
4\beta+(\alpha-2)\epsilon &=0\\
(3\alpha-4)\beta+4\epsilon &=0. %
\end{aligned}
\right.
\]
Either \(\beta=\epsilon=0\) is the only solution, or \(\alpha\in\{-2/3,4\}\) and \(\epsilon=\dfrac{4}{2-\alpha}\beta\). %
\begin{itemize}
\item The case \(\beta=\epsilon=0\) leads to the second family: \(\accol{\phantom{f},\phantom{f}}=\parenth{\phantom{f},\phantom{f}}_\alpha\).
\item The case \(\alpha=-2/3\) and \(\epsilon=3\beta/2\neq 0\) leads to the first family: \(\accol{\phantom{f},\phantom{f}}=\accol{\phantom{f},\phantom{f}}_{3\beta}\).
\item The case \(\alpha=4\) and \(\epsilon=-2\beta\neq 0\) leads to the third family: \(\accol{\phantom{f},\phantom{f}}=\scal{\phantom{f},\phantom{f}}_{\beta}\).
\end{itemize}
Using Propositions~\ref{prop_B} and~\ref{prop_C} we conclude that the only admissible Poisson brackets, up to Poisson modular-isomorphisms are \(\accol{\phantom{f},\phantom{f}}_1\), \(\scal{\phantom{f},\phantom{f}}_1\) and \(\parenth{\phantom{f},\phantom{f}}_\alpha\) for any \(\alpha\in\CC\). Looking at the centres, it is clear that the Poisson algebras \(\left(\fQM,\scal{\phantom{f},\phantom{f}}_1\right)\) and \(\left(\fQM,\accol{\phantom{f},\phantom{f}}_1\right)\) are not Poisson modular isomorphic. Suppose that there exists a Poisson modular isomorphism \(\varphi\) from \(\left(\fQM,\parenth{\phantom{f},\phantom{f}}_\alpha\right)\) to \(\left(\fQM,\accol{\phantom{f},\phantom{f}}_1\right)\). We know (see~\S~\ref{sec_troisdeux}) that %
\[%
\varphi(\E_4)=\E_4,\quad \varphi(\E_6)=\E_6,\quad\text{and}\quad \varphi(\E_2)=\eta\E_2+F %
\]
for some \(\eta\in\CC^*\) and \(F\in\Mod_*\). From \(\varphi\left(\parenth{\E_2,\E_4}_\alpha\right)=\accol{\varphi(\E_2),\varphi(\E_4)}_1\) we obtain %
\[%
\alpha\eta\E_6\E_2+\alpha F\E_6=-\frac{2}{3}\eta\E_6\E_2+\frac{1}{3}\eta\E_4^2+\accol{F,\E_4}_1 %
\]
hence %
\[%
 \begin{aligned}
\alpha&=-\frac{2}{3},\\
& \phantom{=-\frac{2}{3}}%
\end{aligned}
\quad
 \begin{aligned} \frac{1}{3}\eta\E_4^2 &= -\frac{2}{3}F\E_6-\accol{F,\E_4}_1\\
&= -\frac{2}{3}F\E_6+2(\E_4^3-\E_6^2)\frac{\partial F}{\partial\E_6} %
\end{aligned}
\]
by~\eqref{eq_jacob}. We get a contradiction. Replacing \(\accol{\phantom{f},\phantom{f}}_1\) by \(\scal{\phantom{f},\phantom{f}}_1\) we get %
\[%
\alpha=4,\quad \eta\E_4^2= 4F\E_6-2(\E_4^3-\E_6^2)\frac{\partial F}{\partial\E_6} %
\]
and again, we get a contradiction. %
%.............................................
\section{Star products on quasimodular forms}
\subsection{Extension of the first family}
This section is devoted to the proof of Theorem~\ref{thm_A}. We will use the following result of Zagier~\cite[Example 1]{MR1280058}. Let \(A=\bigoplus A_k\) be a commutative graded algebra with a derivation \(\der\) homogeneous of degree \(2\) (i.e. \(\der(A_k)\subset A_{k+2}\)). Let us define, for any \(f\in A_k,\ g\in A_\ell, \ r\geq 0\):
\begin{equation}\label{eq_verZ}
[f,g]_{d,r}=\sum_{i=0}^r(-1)^i\binom{k+r-1}{r-i}\binom{\ell+r-1}{i}\der^i(f)\der^{r-i}(g)\in A_{k+\ell+2r}.
\end{equation}
Then \(A\) equipped with these brackets is a Rankin-Cohen algebra, that  means that all algebraic identities satisfied by the usual Rankin-Cohen brackets on modular forms are also satisfied, in particular those expressing of the associativity of the corresponding star product. We obtain the following result.
\begin{theorem}\label{thm_huit}
The star product defined by %
\[
f\#g=\sum_{n\geq 0}[f,g]_{\der,n}\hslash^n %
\]
defines a formal deformation on \(A\). %
\end{theorem}
In particular, we recover the fact, given by Corollary~\ref{cor_trois}, that \([\phantom{f},\phantom{g}]_{\der,1}\) is a Poisson bracket. Note also that this Theorem can be obtained from Connes \& Moscovici's result cited below (see~\S~\ref{sec_extfamdeux}). %

Let \(a\in\CC\) and \(\der_a\) be the homogeneous derivation of degree \(2\) on \(\fQM\) defined by %
\[%
\der_a(\E_2)=2a\E_2^2-\frac{1}{12}\E_4,\qquad \der_a(\E_4)=4a\E_4\E_2-\frac{1}{3}\E_6,\qquad \der_a(\E_6)=6e\E_6\E_2-\frac{1}{2}\E_4^2. %
\]
A direct computation proves that the two Poisson brackets \([\phantom{f},\phantom{f}]_{\der_a,1}\) and \(\accol{\phantom{f},\phantom{f}}_1\) coincide on generators hence are equal on \(\fQM\). %
\begin{remark}\label{rem_neuf}
A derivation \(\der\) on \(\fQM\) is complex-like if %
\(%
\der\Mod_k^{\leq s}\subset\Mod_{k+2}^{\leq s+1}
\)
for all \(k\) and \(s\). Let \(\pi\) be the derivation on \(\fQM\) defined by linear extension of %
\(%
\pi(f)=kf\E_2 %
\)
for all \(f\in\Mod_k^{\leq\infty}\). The set of complex-like derivations \(\der\) such that \([\phantom{f},\phantom{f}]_{\der,1}=0\) is the vector space of dimension \(1\) over \(\CC\) generated by \(\pi\). Let us define \(\dZ\) on \(\fQM\) by %
\[%
\dZ(f)=\frac{\accol{\Delta,f}_1}{12\Delta}. %
\]
Then, %
\[%
\der_a=\dZ+a\delta. %
\]
This implies in particular that if a complex-like derivation \(\der\) satisfies \([\phantom{f},\phantom{f}]_{\der,1}=\accol{\phantom{f},\phantom{f}}_1\) then \(\der=\der_a\) for some \(a\in\CC\). %
\end{remark}
Point~\ref{item_thmAii}) of Theorem~\ref{thm_A} is obtained by a direct application of Theorem~\ref{thm_huit}. We prove now~\ref{item_thmAiii}). The term of highest degree with respect to \(\E_2\) in \([\E_2,\E_4]_{\der_a,2}\) is \(8a^2\E_4\E_2^3\). This forces \(a=0\). Conversely, if \(a=0\) then \(\der_0\fQM\subset\Mod_*\). For any \(f=f_i\E_2^i\) with \(f_i\in\Mod_*\) we have %
\[%
\der_0(f)=\der_0(f_i)\E_2^i-\frac{1}{12}if_i\E_4\E_2^{i-1} %
\]
hence \(\deg_{\E_2}\der_0(f)\leq\deg_{\E_2}f\) and \(\deg_{\E_2}\der_0^j(f)\leq\deg_{\E_2}f\) for any \(f\in\fQM\) and \(j\geq 0\). This implies that %
\[%
[\Mod_k^{\leq s},\Mod_\ell^{\leq t}]_{\der_0,n}\subset\Mod_{k+\ell+2n}^{\leq s+t}.
\]
%...................................;
\subsection{Extension of the second family}\label{sec_extfamdeux}
The aim of this section is to prove Theorem~\ref{thm_B}. The proof of~\ref{item_thmBi}) is similar to the proof of~\ref{item_thmAi}) in Theorem~\ref{thm_A}. Let \(\mathcal{K} \colon\fQM\to\CC\) a graded-additive map. For any integer \(n\geq 0\), we define a bilinear application \([\phantom{f},\phantom{f}]_{\der,n}^{\mathcal{K}}\) by bilinear extension of%
\[%
[f,g]_{\der,n}^{\mathcal{K}}=\sum_{r=0}^n(-1)^r\binom{\mathcal{K}(f)+n-1}{n-r}\binom{\mathcal{K}(g)+n-1}{r}\der^rf\der^{n-r}g. %
\]
By Corollary~\ref{cor_trois}, we know that \([f,g]_{\der,1}^{\mathcal{K}}\) is a Poisson bracket. %

Let us fix \(\mathcal{K}_\alpha\) to be the linear extension on \(\fQM=\bigoplus_{k}\bigoplus_{s}\Mod_{k}^s\) of %
\begin{equation}\label{eq_defK}%
\mathcal{K}_\alpha(f)=\left(k-(3\alpha+2)s\right)\qquad(f\in\Mod_k^s).
\end{equation}
Let \(\pi_\alpha\) be the derivation on \(\fQM\) defined by \(\pi_\alpha(f)=\mathcal{K}_\alpha(f)f\E_2\) for all \(f\in\fQM\). The set of complex-like derivations such that \([\phantom{f},\phantom{f}]_{\der,1}^{\mathcal{K}_\alpha}=0\) is the vector space of dimension \(1\) over \(\CC\) generated by \(\pi_\alpha\).  %
Define derivations \(\dV\) and \(\delta_{\alpha,b}\) on \(\fQM\) by %
\[%
\dV(f)=\frac{\parenth{\Delta,f}_\alpha}{12\Delta} %
\]
and %
\[%
\delta_{\alpha,b}=\dV+b\pi_\alpha.
\]
Note that \(\dV\) does not depend on \(\alpha\). By comparing the values on the generators, it is immediate that \(\parenth{\phantom{f},\phantom{f}}_\alpha=[\phantom{f},\phantom{f}]_{\delta_{\alpha,b},1}^{\mathcal{K}_\alpha}\).
\begin{remark}
Direct computations show that if \(\der\) is a homogeneous derivation of degree \(2\) and \(\mathcal{K}\) is such that \(\parenth{\phantom{f},\phantom{f}}_\alpha=[\phantom{f},\phantom{f}]_{\der,1}^{\mathcal{K}}\) then we necessarily have \(\mathcal{K}=\mathcal{K}_\alpha\) and \(\der=\delta_{\alpha,b}\) for some \(b\in\CC\). %
\end{remark}

The condition that \([\E_4,\E_6]_{\delta_{\alpha,b},2}^{\mathcal{K}}\) has to be a modular form implies \(b=0\) or \(\alpha=-1/3\). For \(\alpha=-1/3\) condition~\eqref{associativity} for \(\mu_r=[\phantom{f},\phantom{g}]_{\delta_{\alpha,b},r}^{\mathcal{K}}\) and \(n=3\) is not satisfied (this can be done by using computer assistance, for example with \textsf{Sage}~\cite{sage}). We assume then that \(b=0\).

Connes \& Moscovici~\cite[Remark 14]{MR2074985} (see also~\cite[\S~II.2]{Yao_PHD} for a nice presentation of this result) proved that, if \(E\) and \(H\) are two derivations of an algebra \(R\) such that \(HE-EH=E\), then the applications \(\mu_n\colon R\times R\to R\) defined by %
\begin{equation}\label{eq_genCM}%
\mu_n(f,g)=\sum_{r=0}^n\frac{(-1)^r}{r!(n-r)!}[E^r\circ(2H+r)^{<n-r>}(f)]\cdot[E^{n-r}\circ(2H+n-r)^{<r>}(g)] %
\end{equation}
define a formal deformation on \(R\) with the notation %
\[%
F^{<m>}=F\circ(F+1)\circ(F+2)\circ\dotsm\circ(F+m-1). %
\]
Let \(\varpi\) be the derivation defined on \(\fQM\) by  %
\(%
\varpi(f)=\mathcal{K}(f)f. %
\)
Then we have %
\[%
\varpi\circ\delta_{\alpha,0}-\delta_{\alpha,0}\circ\varpi=2\delta_{\alpha,0}. %
\]
 We use Connes \& Moscovici's result with \(E=\delta_{\alpha,0}\) and \(H=\varpi/2\) to obtain %
\[%
\mu_n(f,g)=\sum_{r=0}^n(-1)^r\binom{k-(3\alpha+2)s+n-1}{n-r}\binom{\ell-(3\alpha+2)t+n-1}{r}\delta_{\alpha,0}^r(f)\delta_{\alpha,0}^{n-r}(g). %
\]
This implies Theorem~\ref{thm_B}.

\begin{remark}
We could have apply Connes \& Moscovici's result to extend the first family. Indeed Zagier's result is a consequence of Connes \& Moscovici's one. Let \(d\) be a derivation homogeneous of degree \(2\) of the commutative graded algebra \(A=\bigoplus A_k\). It is obvious that the linear map defined on each \(A_k\) by \(H(f)=\dfrac{k}{2}f\) is a derivation of \(A\). It is also clear that it satisfies \(H\circ d-d\circ H=d\). In particular, for any \(f\in A_k\) and \(g\in A_\ell\) we calculate %
\[%
(2H+r)^{<n-r>}(f)=\frac{(k+n-1)!}{(k+r-1)!}f %
\]
and
\[%
\left(2H+(n-r)\right)^{<r>}(g)=\frac{(\ell+n-1)!}{(\ell+n-r-1)!}g. %
\]
Hence a direct application of formula~\eqref{eq_genCM} gives formula~\eqref{eq_verZ}. %
\end{remark}
%...................................;
\subsection{Extension of the third family}\label{sec_extfamtrois}
We do not extend the third family since, for \(\mu\neq 0\), the bracket \(\scal{\phantom{\E_4},\phantom{\E_6}}_\mu\) has not the shape of a Rankin-Cohen bracket. More precisely, if there exists a function \(\kappa\colon\fQM\to\CC\) and a complex-like derivation \(\delta\) of \(\fQM\) such that %
\[%
\scal{f,g}_\mu=\kappa(f)f\delta(g)-\kappa(g)g\delta(f) %
\]
for all \(f\) and \(g\) in \(\fQM\), then \(\mu=0\). Indeed, assume \(\kappa\) and \(\delta\) exist, then %
\[%
\left\{%
\begin{aligned}
\delta(\E_2)&=A\E_2^2+B\E_4\\
\delta(\E_4)&=C\E_4\E_2+D\E_6\\
\delta(\E_6)&=E\E_6\E_2+F\E_4^2\\
\end{aligned}
\right.
\]
for some complex numbers \(A\), \(B\), \(C\), \(D\), \(E\) and \(F\). Since we know the values of \(\scal{\phantom{\E_4},\phantom{\E_6}}_\mu\) on the generators, we get a system relying  \(A\), \(B\), \(C\), \(D\), \(E\), \(F\), \(\kappa(\E_2)\),  \(\kappa(\E_4)\) and  \(\kappa(\E_6)\). It is not difficult to prove that this system has a solution if and only if \(\mu=0\). %
%.............................................
\newcommand{\etalchar}[1]{$^{#1}$}

%.............................................
\end{document}